\documentclass[reqno,12pt]{amsart}
\usepackage{amsfonts,amsmath,amssymb,hyperref,mathtools,enumitem}
\usepackage{datetime}

\usepackage[margin=1in]{geometry}
\numberwithin{equation}{section}

\newcommand{\Z}{\mathbb{Z}}

\newcommand{\N}{\mathbb{N}}
\newcommand{\R}{\mathbb{R}}
\newcommand{\T}{\mathbb{T}}

\newcommand{\G}{\mathcal{G}}

\providecommand{\ip}[1]{\langle#1\rangle}
\providecommand{\abs}[1]{\left\lvert#1\right\rvert}
\providecommand{\norm}[1]{\left\|#1\right\|}

\newtheorem{theorem}{Theorem}[section]
\newtheorem{lemma}[theorem]{Lemma}
\newtheorem{corollary}[theorem]{Corollary}
\newtheorem{proposition}[theorem]{Proposition}

\theoremstyle{definition}
\newtheorem{definition}[theorem]{Definition}
\newtheorem*{acknowledgements}{Acknowledgements}

\theoremstyle{remark}
\newtheorem{remark}[theorem]{Remark}

\begin{document}

\title{The Surface Quasi-Geostrophic Equation with Random Diffusion}
\author[T. Buckmaster]{Tristan Buckmaster}
\address{Department of Mathematics, Princeton University,  Fine Hall, Princeton, NJ 08544 USA}
\email{buckmaster@math.princeton.edu.}
\author[A. Nahmod]{Andrea Nahmod}
\address{Department of Mathematics, University of Massachusetts, 710 N.\ Pleasant Street, Am\-herst, MA 01003 USA}
\email{nahmod@math.umass.edu}
\author[G. Staffilani]{Gigliola Staffilani}
\address{Department of Mathematics,  Massachusetts Institute of Technology,  77 Massachusetts Ave,  Cambridge,  MA 02139-4307 USA}
\email{gigliola@math.mit.edu}
\author[K. Widmayer]{Klaus Widmayer}
\address{Institute of Mathematics, EPFL, Station 8, 1015 Lausanne, Switzerland}
\email{klaus.widmayer@epfl.ch}

\thanks{T.B.\ was partially supported by NSF-DMS-1600868.  A.N.\ was partially supported by NSF-DMS-1463714. G.S.\ was partially supported by NSF-DMS-1362509 and NSF-DMS-1462401,  the Simons Foundation and the John Simon Guggenheim Foundation. K.W.\ was partially supported by SNSF Grant 157694.}

\begin{abstract}
Consider the surface quasi-geostrophic equation with random diffusion, white in time. We show global existence and uniqueness in high probability for the associated Cauchy problem satisfying a Gevrey type bound. This article is inspired by recent work of Glatt-Holtz and Vicol \cite{MR3161482}.
\end{abstract}

\maketitle

\section{Introduction}
In this article we describe an instance of \emph{regularization via noise} for a nonlinear partial differential equation, namely the \emph{surface quasi-geostrophic (SQG)} equation. The principle underlying regularization via noise is that for certain systems, ill-posedness or blow up may be rare, and thus considering additional noise could potentially transform a system that is deterministically ill-posed into a system that is probabilistically well-posed. 
This is in contrast to the idea that noise introduces irregularities, which can render control of the systems more difficult or weaker than in the deterministic case. 

Regularization via noise is well known to occur in some examples of nonlinear ordinary equations (e.g.\ \cite{MR1202675}). For (nonlinear) partial differential equations less is known, and we refer the reader to the book of Flandoli \cite{MR2796837} for an introduction and some overview. In general, depending on the circumstances noise may regularize \cite{DEBUSSCHE2011363,Chouk,MR2593276} or singularize \cite{MR1947366, debouard2005} a given system. 

Our object of study is an SQG equation for $\theta:\R\times\T^2\to\R$ with a random in time diffusion (or multiplicative noise), written as an It\^{o} stochastic integral:
\begin{equation}\label{eq:SQG-prob}
 d\theta +R^\perp\theta\cdot\nabla\theta dt=\nu\abs{\nabla}^s\theta\, dW_t,\quad \theta(0)=\theta_0.
\end{equation}
Here $R^\perp$ is the vector of Riesz transforms $R^\perp=(-R_2,R_1)$ and $W_t$ is a standard Wiener process in time. In particular, we note that the term on the right hand side of \eqref{eq:SQG-prob} does not have a definite sign. The index $s$ of differentiability with $0<s\leq 1$ and the constant of diffusion $\nu>0$ are regarded as fixed. 

This is a stochastic version of the well-known SQG equation, which under the quasi-geostrophic approximation describes the evolution of the surface temperature $\theta$ for rapidly rotating geophysical fluids with small Rossby and Ekman numbers  \cite{Pe1982}.   The inviscid SQG equation
\begin{equation}\label{e:inviscid}
 \partial_t\theta+R^\perp\theta\cdot\nabla\theta=0
\end{equation}
 has also been suggested as a mathematical model to study singularity formation for the 3D Euler equations \cite{CoMaTa1994}. 
 
Global well-posedness for the inviscid SQG equation \eqref{e:inviscid} is a major open problem in the field of mathematical fluid dynamics \cite{MR2338368}. In fact, only local well-posedness is known \cite{CoMaTa1994} and except for specific examples, the only solutions of \eqref{e:inviscid} that are known to exist globally are weak solutions \cite{Resnick95,Ma2008}. However, we note that certain classes of weak solutions are known to be non-unique \cite{BSV-SQG}. 

Contrasting strongly with this, in the present article we will show that for initial data satisfying a certain bound depending on the constants $s$ and $\nu$, the \emph{Cauchy problem to the stochastic SQG \eqref{eq:SQG-prob} has a unique global solution with high probability}. 

Our main result can then be stated as follows:
\begin{theorem}\label{thm:main_short}
 Let $\sigma\in (\frac{1}{s},2)$ and $\theta_0\in \dot{H}^{\sigma s}$ be given with $\hat{\theta}_0(0)=0$ and
 \begin{equation}\label{eq:init_cond0}
  \norm{e^{\alpha \abs{\nabla}^{s}}\,\theta_0}_{\dot{H}^{\sigma s}}=:E<\frac{\nu^2}{2}
 \end{equation}
 for some $\alpha>0$.
 
 Then there exists $\beta=\beta(\nu,E)<E$ such that for $\epsilon\in (0,\alpha)$ we obtain a pathwise unique, global solution $\theta\in C^0\dot{H}^{\sigma s}$ of \eqref{eq:SQG-prob} with probability $$p(\alpha,\epsilon,\nu,E)\geq 1-e^{-2\frac{(\alpha-\epsilon)\beta}{\nu^2}}.$$

 In particular, for any $p_0\in (0,1)$ and $E>0$ given, there exist $\nu>0$ and $\alpha>0$ such that if initial data $\theta_0$ satisfy \eqref{eq:init_cond0} for these values we obtain a pathwise unique global solution with likelihood at least $p_0$.
\end{theorem}
For the more detailed statement of the result we refer the reader to Theorem \ref{thm:main_full} and its Corollary \ref{cor:main}.

The proof of Theorem \ref{thm:main_short} relies on a transformation which converts the stochastic SQG equation \eqref{eq:SQG-prob} into a new  \emph{pathwise deterministic equation with diffusion}.  
This transformation is inspired by the use of geometric Brownian motion in \cite{MR3161482}.

One might also compare Theorem \ref{thm:main_short} to the case of the viscous SQG equation
\begin{equation}\label{e:viscous SQG}
 \partial_t\theta+R^\perp\theta\cdot\nabla\theta=-\nu\abs{\nabla}^s\theta
\end{equation}
for $s>0$, where for small data, global existence is known \cite{MR1962043,MR2084005}.\footnote{For the critical, and subcritical cases $s\geq\frac{1}{2}$ unconditional global existence is known \cite{MR1709781,KiNaVo2007,CaVa2010}.} However, we highlight that in contrast to \eqref{e:viscous SQG}, the righthand side of \eqref{eq:SQG-prob} does \emph{not} have a sign and as such \emph{does not a-priori regularize the solution}. 

The potential for the noise to create singularities in the solution motivates the use of Gevrey spaces in order to account for potential loss of regularity. As such, we will employ energy techniques in the spirit of \cite{MR1026858}  (see also \cite{MR1719695}) on the transformed equation.

\subsection*{Overview of the Article}
We start the presentation by demonstrating in Section \ref{sec:transf} how the stochastic SQG equation \eqref{eq:SQG-prob} can be transformed into a pathwise deterministic equation with diffusion. This motivates the framework and techniques of the rest of the paper, and we include the relevant setup in that section. In particular, the Gevrey type spaces we work in are defined there.

The following sections work with the transformed equation \eqref{eq:SQG-det}, and establish first local well-posedness (Proposition \ref{prop:local} in Section \ref{sec:lwp}) and then energy estimates (Proposition \ref{prop:en_est} in Section \ref{sec:en_est}) in Gevrey spaces. Since pathwise the transformed equation is deterministic, these arguments are not probabilistic in nature: thanks to the diffusion present in \eqref{eq:SQG-det} we can obtain local well-posedness using a contraction argument, and the energy estimates proceed as is common for such equations.

A probabilistic estimate for Brownian motion with drift (Lemma \ref{lem:drift}) informs our global result in Section \ref{sec:global}. Taking this into account, in our Theorem \ref{thm:main_full} we construct global solutions (for the transformed equation) by an iteration that combines the local existence theory and the energy estimates. Here one has to carefully track the precise structure of the estimates and regularities (for a more precise explanation see the ``strategy of proof'' on page \pageref{pf:strat-proof}).

Finally, Corollary \ref{cor:main} translates this result for the transformed equation \eqref{eq:SQG-det} back to the original equation \eqref{eq:SQG-prob}, thereby establishing the claim in Theorem \ref{thm:main_short}.

\begin{acknowledgements}
The authors would like to thank the hosts and organizers of the workshop ``Mathematical questions in wave turbulence theory'' in May 2017 at the American Institute of Mathematics, during which the initial plans for this project formed.
\end{acknowledgements}
\section{Transformation of the Equation}\label{sec:transf}
Recall the equation \eqref{eq:SQG-prob}
\begin{equation*}
 d\theta +R^\perp\theta\cdot\nabla\theta dt=\nu\abs{\nabla}^s\theta\, dW_t,\quad \theta(0)=\theta_0.
\end{equation*}
Setting $\Gamma(t):=e^{-\nu W_t\abs{\nabla}^s}$ we have $d\Gamma=-\nu\abs{\nabla}^s\Gamma dW_t+\frac{1}{2}\nu^2\abs{\nabla}^{2s}\Gamma dt$, and thus by It\^{o}'s product rule and commutativity of $\Gamma$ with $\abs{\nabla}^s$ we arrive at the following equation for $u(t):=\Gamma(t)\theta(t)$:
\begin{equation*}
  du=d\Gamma\cdot\theta +\Gamma\cdot d\theta+d\Gamma\cdot d\theta=-\Gamma\left(R^\perp\Gamma^{-1}u\cdot\nabla\Gamma^{-1}u\right)dt-\frac{1}{2}\nu^2\abs{\nabla}^{2s}u dt,
\end{equation*}
i.e.
\begin{equation}\label{eq:SQG-det}
\begin{cases}&\partial_t u+\Gamma\left(R^\perp\Gamma^{-1}u\cdot\nabla\Gamma^{-1}u\right)=-\frac{1}{2}\nu^2\abs{\nabla}^{2s}u,\\ 
 &u(0)=u_0=\theta_0.
\end{cases} 
\end{equation}
Remarkably, this is a \emph{pathwise deterministic} equation, i.e.\ for a fixed instance of the Brownian motion $t\mapsto W_t(\omega)$ we have a deterministic equation. In what follows, we will work on \eqref{eq:SQG-det} and prove first local existence, energy estimates and finally global existence of solutions in this pathwise deterministic setting. It is only at the very end of Section \ref{sec:global} (see Corollary \ref{cor:main}) that we come back to the original equation \eqref{eq:SQG-prob}.

\begin{remark}
 Upon integration of the equation \eqref{eq:SQG-det}, we obtain $\int u(t) dx=const$, so that we may -- and henceforth will -- assume without loss of generality that $\int udx=\hat{u}(0)=0$.\footnote{In particular, homogeneous and inhomogeneous Sobolev spaces on $\T^2$ agree and we have canonical inclusions $H^a(\T^2)\subset H^b(\T^2)$ with $a>b$ for such functions.}
\end{remark}

\subsection{Gevrey Setup}
In what follows we will use the notation $B(f,g)$ for the bilinear term
\begin{equation*}
 B(f,g):=\Gamma\left(R^\perp\Gamma^{-1}f\cdot\nabla\Gamma^{-1}g\right).
\end{equation*}
We write equation \eqref{eq:SQG-det} schematically as
\begin{equation}\label{eq:SQG-det2}
 \partial_t u +B(u,u)=-\frac{1}{2} \nu^2Au
\end{equation}
with $A:=\abs{\nabla}^{2s}$. Following Foias-Temam \cite{MR1026858} we proceed to work with $L^2$ based norms weighted with $e^{\phi(t) A^{1/2}}A^{m}$, where $m\in\{\frac{1}{2},1\}$. Here $\phi(t)$ needs to be suitably chosen, and we will always consider it to be of the form
\begin{equation}\label{eq:phi-def}
 \phi(t)=\alpha+\beta t\text{ with }\alpha>0 \text{ and }0<\beta<\frac{\nu^2}{2},
\end{equation}
or shifts $\phi(t-T_0)$ thereof. The meaning and the conditions for the parameters $\alpha,\beta$ will become clear later (see Propositions \ref{prop:local} and \ref{prop:en_est}), but for now let us emphasize the upper bound $\beta<\frac{\nu^2}{2}$.

We end this section with the definition of the Gevrey spaces that we will construct our solutions in:
\begin{definition}\label{def:gevrey}
 Let $\gamma>0$ and $\varphi$ be a continuous real function. Then for a map $w:[0,T]\times\T^2\to\R$ we let
 \begin{equation}
  \norm{w(t)}_{\G^{\gamma}_{\varphi(t)}}:=\norm{e^{\varphi(t)A^{1/2}}w(t)}_{\dot{H}^{\gamma s}}
 \end{equation}
 and define the corresponding Gevrey space as
 \begin{equation}
  \G^{\gamma}_\varphi:=\{f\in\dot{H}^{\gamma s}:\,\norm{f}_{\G^{\gamma}_\varphi}<\infty\}.
 \end{equation}
From this one obtains the combined space-time spaces
 \begin{equation}
  C^0_T\G^\gamma_\varphi:=\{f\in C^0([0,T],\G^\gamma_\varphi):\,\norm{f}_{C^0_T\G^\gamma_\varphi}<\infty\}.
 \end{equation}
\end{definition}

\section{Local Well-Posedness}\label{sec:lwp}
In this section we discuss the construction of solutions to \eqref{eq:SQG-det} via the contraction mapping principle.

\begin{proposition}\label{prop:local}
 Let $T_0>0$ and $I\subset\R$ be an open interval with $T_0\in I$ as well as $\phi(t)-\nu W_t\geq 0$ for $t\in I$. Fix $\sigma>0$ with $\sigma\in(\frac{1}{s},2)$. Then for mean zero initial data $u_0$ with $e^{\phi(T_0) A^{1/2}}u_0\in\dot{H}^{\sigma s}$ there exist $T>0$ and a unique solution $u\in C^0(I\cap [T_0,T_0+T],\G^\sigma_\phi)$ of \eqref{eq:SQG-det}.
 
 Here $T=T(\|e^{\phi(T_0) A^{1/2}}u_0\|_{\dot{H}^{\sigma s}})$ depends only on the norm $\|e^{\phi(T_0) A^{1/2}}u_0\|_{\dot{H}^{\sigma s}}=\norm{u_0}_{\G^\sigma_{\phi(T_0)}}$ of the initial data.
\end{proposition}

\begin{remark}[On the role of $s$]
 As can be seen easily from the arguments below, for reasons of concavity it is important in our arguments for finding solutions of \eqref{eq:SQG-det} that $s\leq 1$: in fact, this is what allows us to absorb the propagators $e^{-\nu W_t\abs{\nabla}^s}$ using the exponential weights of our Gevrey spaces.
 
 Moreover, a finer analysis of the bilinear terms shows that in case $s>\frac{2}{3}$ one may actually choose $\sigma=1$ in the above proposition and in the remaining arguments of the article, so that Theorems \ref{thm:main_short} and \ref{thm:main_full} hold in that case as well. However, since we are already working in Gevrey spaces with lots of regularity, for the sake of clarity of the presentation we have chosen not to emphasize this point.
\end{remark}

\begin{remark}[On the probability of the assumption]
 Taking $T_0=0$ for simplicity, with this result one obtains a local solution for any given continuous $\phi$ with $\phi(0)=\alpha$. In particular, already for a simple choice like $\phi(t)=\alpha$ the condition that $\phi(t)-\nu W_t\geq 0$ on $[0,T]$ will be satisfied for all Brownian paths except for exponentially few: One recalls that for $\kappa>0$
 $$\mathbb{P}[\sup_{t\in[0,T]}W_t>\kappa]=\frac{2}{\sqrt{2\pi}}\int_{\frac{\kappa}{\sqrt{T}}}^\infty e^{-\frac{y^2}{2}}dy.$$
 For our specific choice of $\phi(t)=\alpha+\beta t$ as in \eqref{eq:phi-def}, the condition that $\phi(t)-\nu W_t\geq 0$ will later be seen to hold \emph{globally} on an exponentially large set (see Lemma \ref{lem:drift}).
\end{remark}

\begin{proof}[Proof of Proposition \ref{prop:local}]
Without loss of generality we assume that $T_0=0$ and write $\phi(t)=\alpha+\beta t$ as in \eqref{eq:phi-def}.\\
Inspired by Duhamel's formula for a solution of \eqref{eq:SQG-det}, let us define the map $\Phi$ as
\begin{equation*}
 \Phi(u)(t):=e^{-\frac{t}{2}\nu^2 A}u_0+\underbrace{\int_0^t e^{-\frac{t-\tau}{2}\nu^2 A}B(u,u)(\tau)d\tau}_{=:Q(u,u)}.
\end{equation*}
Since $\beta<\frac{\nu^2}{2}$ we have that
\begin{equation*}
 \norm{e^{\phi(t)A^{1/2}}e^{-\frac{t}{2}\nu^2 A}u_0}^2_{\dot{H}^{\sigma s}}=\sum_{k\in\Z^2\setminus\{0\}}\abs{k}^{2s\sigma}\abs{e^{\phi(t)\abs{k}^s-\nu^2\frac{t}{2}\abs{k}^{2s}}\hat{u_0}(k)}^2\leq \norm{e^{\alpha A^{1/2}}u_0}_{\dot{H}^{\sigma s}}^2.
\end{equation*}
For the bilinear term it is convenient to mod out by the linear flow and normalize by switching to the variable $v:=\abs{\nabla}^{\sigma s}e^{\phi(t)A^{1/2}}u$, so that
\begin{equation}\label{eq:u->v}
 \hat{u}(t,j)=e^{-\phi(t)\abs{j}^s}\abs{j}^{-\sigma s}\hat{v}(t,j),
\end{equation}
such that $\norm{\hat{v}(t,j)}_{\ell^2_j}=\norm{v(t)}_{L^2}=\norm{e^{\phi(t)A^{1/2}}u(t)}_{\dot{H}^{\sigma s}}$. 

The nonlinearity then reads
\begin{equation*}
 \widehat{B(u,u)(t,k)}=\sum_{j\in\Z^2\setminus\{0\}}\frac{(k-j)\cdot j^\perp}{\abs{k-j}^{\sigma s}\abs{j}^{1+\sigma s}}e^{-\nu W_t[\abs{k}^s-\abs{k-j}^s-\abs{j}^s]}e^{-\phi(t)[\abs{k-j}^s+\abs{j}^s]}\hat{v}(k-j)\hat{v}(j),
\end{equation*}
and hence we obtain
\begin{equation*}
\begin{aligned} 
 &\norm{e^{\phi(t)A^{1/2}}Q(u,u)}^2_{\dot{H}^{\sigma s}}\leq\int_0^t\sum_{k\in\Z^2\setminus\{0\}}\abs{e^{\phi(t)\abs{k}^s-\nu^2\frac{t-\tau}{2}\abs{k}^{2s}}\abs{k}^{\sigma s}\widehat{B(u,u)(\tau,k)}}^2\,d\tau\\
 &\leq\int_0^t\hspace{-1mm}\sum_{k\in\Z^2\setminus\{0\}}\abs{e^{\beta (t-\tau)\abs{k}^s-\nu^2\frac{t-\tau}{2}\abs{k}^{2s}}\abs{k}^{\sigma s} \hspace{-3mm}\sum_{j\in\Z^2\setminus\{0\}}m(k,j)e^{(\phi(\tau)-\nu W_\tau)[\abs{k}^s-\abs{k-j}^s-\abs{j}^s]}\hat{v}(k-j)\hat{v}(j)}^2 d\tau
\end{aligned}
\end{equation*}
after using that $\phi(t)-\phi(\tau)=\beta (t-\tau)$, and having introduced the notation
$$m(k,j):=\frac{(k-j)\cdot j^\perp}{\abs{k-j}^{\sigma s}\abs{j}^{1+\sigma s}}.$$
We observe now that $\abs{k}^s\leq\abs{k-j}^s+\abs{j}^s$ when $0<s\leq 1$, and invoke our assumption that $\phi(t)-\nu W_t\geq 0$ on $[0,T]$ to drop the second exponential. Furthermore we note that since $\beta<\frac{\nu^2}{2}$ there exists $\mu>0$ such that $\beta (t-\tau)\abs{k}^s-\nu^2\frac{t-\tau}{2}\abs{k}^{2s}\leq -\mu(t-\tau)\abs{k}^{2s}$ for all $k\in\Z^2\setminus\{0\}$. We combine this with $\abs{k}^{\sigma s}=\left((t-\tau)\abs{k}^{2s}\right)^{\frac{\sigma}{2}}(t-\tau)^{-\frac{\sigma}{2}}$ to see that there exists $C_\mu>0$ such that
\begin{equation*}
 e^{-\mu(t-\tau)\abs{k}^{2s}}\abs{k}^{\sigma s}\leq \frac{C_\mu}{(t-\tau)^{\frac{\sigma}{2}}}
\end{equation*}
for all $k\in\Z^2\setminus\{0\}$. A standard convolution estimate then yields
\begin{equation}\label{eq:conv_est}
\begin{aligned}
 \norm{e^{\phi(t)A^{1/2}}Q(u,u)}_{\dot{H}^{\sigma s}}&\leq\int_0^t\left(\sum_{k\in\Z^2\setminus\{0\}}\abs{e^{-\mu(t-\tau)\abs{k}^{2s}}\abs{k}^{\sigma s}\sum_{j\in\Z^2\setminus\{0\}}m(k,j)\hat{v}(k-j)\hat{v}(j)}^2\right)^{\frac{1}{2}}d\tau\\
 &\leq\int_0^t\frac{C_\mu}{(t-\tau)^{\frac{\sigma}{2}}}d\tau\cdot \sup_{0\leq\tau\leq t}\norm{\frac{1}{\abs{i}^{\sigma s-1}}\hat{v}(\tau,i)}_{\ell^2_i}\norm{\frac{1}{\abs{j}^{\sigma s}}\hat{v}(\tau,j)}_{\ell^1_j}.
\end{aligned}
\end{equation}
By construction $\sigma s>1$, so we have $\norm{\frac{1}{\abs{i}^{\sigma s-1}}\hat{v}(\tau,i)}_{\ell^2_i}\leq\norm{v(\tau)}_{L^2}$ and also
\begin{equation*}
 \norm{\frac{1}{\abs{j}^{\sigma s}}\hat{v}(\tau,j)}_{\ell^1_j}\leq\norm{\frac{1}{\abs{j}^{\sigma s}}}_{\ell^2_j}\norm{\hat{v}(\tau,j)}_{\ell^2_j}\lesssim\norm{v(\tau)}_{L^2}.
\end{equation*}
Recalling the definition of $v$ in \eqref{eq:u->v} we have thus shown that
\begin{equation*}
 \norm{e^{\phi(t)A^{1/2}}Q(u,u)}_{\dot{H}^{\sigma s}}\leq C_\mu t^{1-\frac{\sigma}{2}}\sup_{0\leq\tau\leq t}\norm{e^{\phi(\tau)A^{1/2}}u(\tau)}_{\dot{H}^{\sigma s}}^2.
\end{equation*}

Recalling Definition \ref{def:gevrey} of the norm 
\begin{equation}
 \norm{u(t)}_{\G^\sigma_\phi}=\norm{e^{\phi(t)A^{1/2}}u(t)}_{\dot{H}^{\sigma s}},\text{ with } \norm{u(t)}_{C^0_T\G^\sigma_\phi}=\sup_{t\in [0,T]}\norm{e^{\phi(t)A^{1/2}}u}_{\dot{H}^{\sigma s}},
\end{equation}
we can write our result so far as the mapping property
\begin{equation}\label{eq:Phi-selfmap}
 \norm{\Phi(u)}_{C^0_T\G^\sigma_\phi}\leq \norm{e^{\alpha A^{1/2}}u_0}_{\dot{H}^{\sigma s}}+C_\mu T^{1-\frac{\sigma}{2}}\norm{u}_{C^0_T\G^\sigma_\phi}^2.
\end{equation}

One proceeds in direct analogy to show that $\Phi$ is also a contraction on a suitably small ball around the origin in the space 
$$C^0_T\G^\sigma_\phi=\{f\in C^0([0,T],\G^\sigma_\phi):\,\norm{f}_{C^0_T\G^\sigma_\phi}<\infty\},$$ 
i.e.\ that
\begin{equation}\label{eq:Phi-contract}
 \norm{\Phi(u)-\Phi(w)}_{C^0_T\G^\sigma_\phi}\leq C_\mu T^{1-\frac{\sigma}{2}}\left(\norm{u}_{C^0_T\G^\sigma_\phi}+\norm{w}_{C^0_T\G^\sigma_\phi}\right)\norm{u-w}_{C^0_T\G^\sigma_\phi}.
\end{equation}

As usual, together \eqref{eq:Phi-selfmap} and \eqref{eq:Phi-contract} demonstrate the existence of a solution in $C^0_T\G^\sigma_\phi$ for $T>0$ small enough. They also show that $T$  depends only on the size of the initial data $\norm{e^{\alpha A^{1/2}}u_0}_{\dot{H}^{\sigma s}}$ as claimed.
\end{proof}

\section{Energy Estimates}\label{sec:en_est}
In this section we obtain a priori estimates for the local solutions to \eqref{eq:SQG-det} constructed in the previous section.
\begin{proposition}\label{prop:en_est}
 Assume  $\sigma\in (\frac{1}{s},2)$ and let $u$ be a local solution on $[T_0,T]$ of \eqref{eq:SQG-det} or \eqref{eq:SQG-det2} that is bounded in $C^0([T_0,T],\G^{\sigma+1}_\phi)$. Assuming that $\phi(t)-\nu W_t\geq 0$ with $\phi(t)$ as in \eqref{eq:phi-def},and that
 \begin{equation}
 \norm{e^{\phi(T_0) A^{1/2}} u(T_0)}_{\dot H^{\sigma s}}\leq \frac{\nu^2}{2}-\beta,
 \end{equation}
 for $t\in[T_0,T]$ we have 
 \begin{equation}
  \norm{e^{\phi(t)A^{1/2}} u(t)}_{\dot H^{\sigma s}}\leq \norm{e^{\phi(T_0) A^{1/2}} u(T_0)}_{\dot H^{\sigma s}}\,.
 \end{equation}
 \end{proposition}

\begin{proof}
 As in standard energy estimates we test the equation with $e^{\phi(t)A^{1/2}} A^\sigma u$ to get
 \begin{equation*}
  \ip{e^{\phi(t)A^{1/2}} \partial_t u,e^{\phi(t)A^{1/2}} A^\sigma u}+\ip{e^{\phi(t)A^{1/2}} B(u,u),e^{\phi(t)A^{1/2}} A^\sigma u}=-\frac{1}{2}\nu^2\norm{e^{\phi(t)A^{1/2}} u}^2_{\dot H^{(\sigma+1)s}}.
 \end{equation*}
 Now
 \begin{equation*}
 \begin{aligned}
  \ip{e^{\phi(t)A^{1/2}} \partial_t u,e^{\phi(t)A^{1/2}} A^\sigma u}&=\frac{1}{2}\frac{d}{dt}\norm{e^{\phi(t)A^{1/2}} u}^2_{\dot H^{\sigma s}}-\beta\ip{e^{\phi(t)A^{1/2}} A^{1/2} u,e^{\phi(t)A^{1/2}} A^\sigma u}\\
  &\geq \frac12\frac{d}{dt}\norm{e^{\phi(t)A^{1/2}} u}^2_{\dot H^{\sigma s}}-\beta\norm{e^{\phi(t)A^{1/2}} u}^2_{\dot H^{(\sigma+1)s}}
 \end{aligned}
 \end{equation*}
since by construction $\phi'(t)=\beta$.
 
 Combining this with the bilinear estimate
 \begin{equation*}
 \ip{e^{\phi(t)A^{1/2}} B(u,u),e^{\phi(t)A^{1/2}} A^\sigma u}\leq \norm{e^{\phi(t)A^{1/2}} u}_{\dot H^{\sigma s}}\norm{e^{\phi(t)A^{1/2}} u}^{2}_{\dot H^{(\sigma+1)s}}
 \end{equation*}
 from Lemma \ref{lem:bilin} below yields
 \begin{equation}\label{eq:deriv_bound}
  \frac{d}{dt}\norm{e^{\phi(t)A^{1/2}} u}^2_{\dot H^{\sigma s}}+\left(\nu^2-2\beta-2 \norm{e^{\phi(t)A^{1/2}} u}_{\dot H^{\sigma s}}\right)\norm{e^{\phi(t)A^{1/2}} u}^2_{\dot H^{(\sigma+1)s}}\leq0\,.
 \end{equation}
 Thus so long as
 \[\norm{e^{\alpha A^{1/2}} u(0)}_{\dot H^{\sigma s}}\leq \frac{\nu^2}{2}-\beta\]
 then
 \[\norm{e^{\phi(t)A^{1/2}} u}_{\dot H^{\sigma s}}\leq \norm{e^{\alpha A^{1/2}} u(0)}_{\dot H^{\sigma s}}\,.\]
\end{proof}

\begin{lemma}\label{lem:bilin}
 Assuming that $\phi(t)-\nu W_t\geq 0$ for all $t\in [T_0,T]$ we have the following bilinear bounds on $[T_0,T]$:
 \begin{equation}
   \abs{\ip{e^{\phi(t)A^{1/2}}B(w,w),e^{\phi(t)A^{1/2}}Aw}_{L^2}}\lesssim \norm{e^{\phi(t)A^{1/2}} w}_{\dot H^{\sigma s}}\norm{e^{\phi(t)A^{1/2}} w}^2_{\dot H^{(\sigma+1)s}}.
  \end{equation}
 \end{lemma}

\begin{proof} 
 Taking Fourier transforms we have
 \begin{equation}
 \begin{aligned}
  &\abs{\ip{e^{\phi(t)A^{1/2}}B(u,v),e^{\phi(t)A^{1/2}}A^\sigma w}_{L^2}}\\
  &=\abs{\sum_{j,k\in\Z^2}e^{\phi(t)\abs{k+j}^s}e^{-\nu W_t[\abs{k+j}^s-\abs{j}^s-\abs{k}^s]}\frac{j^\perp\cdot k}{\abs{j}}\hat{u}(j)\hat{v}(k)e^{\phi(t)\abs{k+j}^s}\abs{k+j}^{2\sigma s}\overline{\hat{w}(k+j)}}\\
  &=\abs{\sum_{j,k\in\Z^2}e^{[\phi(t)-\nu W_t][\abs{k+j}^s-\abs{j}^s-\abs{k}^s]}\frac{j^\perp\cdot k}{\abs{j}}(e^{\phi(t)\abs{j}^s}\hat{u}(j))(e^{\phi(t)\abs{k}^s}\hat{v}(k))e^{\phi(t)\abs{k+j}^s}\abs{k+j}^{2\sigma s}\overline{\hat{w}(k+j)}}\\
  &\leq \sum_{j,k\in\Z^2}\abs{\frac{j^\perp\cdot k}{\abs{j}}(e^{\phi(t)\abs{j}^s}\hat{u}(j))(e^{\phi(t)\abs{k}^s}\hat{v}(k))e^{\phi(t)\abs{k+j}^s}\abs{k+j}^{2\sigma s}\overline{\hat{w}(k+j)}}\\
  &\leq \norm{R^\perp e^{\phi(t)A^{1/2}} u\cdot \nabla e^{\phi(t)A^{1/2}} v}_{\dot{H}^{(\sigma-1)s}}\norm{e^{\phi(t)A^{1/2}}w}_{\dot{H}^{(\sigma+1)s}},
 \end{aligned}
 \end{equation}
 where we used that by assumption $\phi(t)-\nu W_t\geq 0$ and the fact that for $0\leq s\leq 1$, $j,k\in\Z^2$ the inequality $\abs{k+j}^s\leq\abs{j}^s+\abs{k}^s$ holds.
 
 Now it suffices to use H\"older's inequality, the Sobolev embedding and interpolation to obtain the claim: Note that one has the estimate $$\norm{R^\perp f\cdot\nabla f}_{\dot{H}^{(\sigma-1)s}}\lesssim \norm{f}_{\dot H^{\sigma s}}\norm{f}^{2}_{\dot H^{(\sigma+1)s}},$$ which gives the claim upon setting $f=e^{\phi(t)A^{1/2}} w$.
\end{proof}

\section{Global Solutions}\label{sec:global}
The progression from local solutions to global ones is now a matter of suitably combining the energy estimates with the local existence argument. Hereby it is important to know about the validity of the bound $\phi(t)-\nu W_t\geq 0$. To this end we remind the reader of the following standard result for Brownian motion with a drift:
\begin{lemma}\label{lem:drift}
For $\alpha ,\beta>0$, the probability that the process $W_t-\beta t > \alpha$ in finite time is given by $e^{-2\alpha \beta}$.
\end{lemma}
\begin{proof}
 See \cite[Proposition 6.8.1]{MR1181423}, for example.
\end{proof}

This allows us to deduce our main result:
\begin{theorem}\label{thm:main_full}
 Let $\sigma\in(\frac{1}{s},2)$, and assume that we are given initial data $u_0$ satisfying
 \begin{equation}\label{eq:init_cond1}
  \norm{u_0}_{\G^{\sigma}_{\alpha+\epsilon}}=\norm{e^{(\alpha+\epsilon) A^{1/2}}u_0}_{\dot{H}^{\sigma s}}\leq \frac{\nu^2}{2}-\beta
 \end{equation}
 for some $\alpha>0$, $\epsilon>0$ and $0<\beta<\frac{\nu^2}{2}$. Define $\phi(t):=\alpha+\beta t$.
 
 Then with probability at least $1-e^{-2\frac{\alpha \beta}{\nu^2}}$ there exists a pathwise\footnote{i.e.\ for each fixed path $t\mapsto W_t(\omega)$ we have a unique solution.} unique global solution $u\in C^0([0,\infty),\G^\sigma_\phi)$ to \eqref{eq:SQG-det},
 \begin{equation*}
  \begin{cases}&\partial_t u+\Gamma\left(R^\perp\Gamma^{-1}u\cdot\nabla\Gamma^{-1}u\right)=-\frac{1}{2}\nu^2\abs{\nabla}^{2s}u,\\ 
  &u(0)=u_0.
  \end{cases} 
 \end{equation*}
 
 Moreover, the mapping
 \begin{equation}
  t\mapsto \norm{u(t)}_{\G^\sigma_\phi}=\norm{e^{\phi(t)A^{1/2}}u(t)}_{\dot{H}^{\sigma s}}
 \end{equation}
 is pathwise monotonically decreasing.
\end{theorem}

\begin{remark}[The meaning of the parameters $\alpha,\beta,\nu$]
 We may phrase the theorem in terms of the size of the initial condition as follows: Let $E>0$ be such that $\norm{u_0}_{\G^{\sigma}_{\alpha+\epsilon}}\leq E$. Then for $\beta<\frac{\nu^2}{2}-E$ and $\alpha>0$ we let $\phi(t)=\alpha+\beta t$ and obtain a global solution in $C^0\G^\sigma_\phi$ with probability $1-e^{-2\frac{\alpha \beta}{\nu^2}}$. 
 
 Put differently, for any size $E$ of initial data and any probability $p\in[0,1)$, for a sufficiently large parameter $\nu$ and sufficiently smooth initial data (i.e.\ $\alpha$ sufficiently large) we obtain global solutions with probability at least $p$.
 
 \emph{More precisely, since we may always choose $\beta\sim\nu^2$, the diffusion constant $\nu$ governs the rate of growth of the radius of analyticity, whereas the level of regularity $\alpha$ increases the likelihood of obtaining a global solution. Note that with a fixed $\alpha$, varying $\nu$ (and $\beta<\frac{\nu^2}{2}$) alone, we can only guarantee global solutions with probability $1-e^{-\alpha}$.}
 
 We note that by construction of the space $\G^\sigma_\phi$, the space regularity (or Gevrey index) of the solution actually increases with the passage of time.
\end{remark}

\begin{proof}[Strategy of Proof]\label{pf:strat-proof}
The above Lemma \ref{lem:drift} for Brownian motion with drift will allow us to work pathwise, where we combine our local existence and energy estimates. The challenge here is that the energy estimates require finiteness of a ``higher'' norm than the one for which they give monotonicity, i.e.\ we will have to control two levels of regularity. More precisely, in order to show that the $\G^\sigma_\phi$ norm is monotonically decreasing on a given time interval, we have to guarantee that the $\G^{\sigma+1}_\phi$ norm remains finite on that same interval. One sees readily that this cannot be done using just the local well-posedness theory, since a priori the higher energy norm will only be controlled for a shorter amount of time. 

The key point here is that -- thanks to our Gevrey setting -- we can achieve what we need by just an $\epsilon$ loss in regularity: note that we have the embeddings
\begin{equation}\label{eq:G-emb-gen}
 \norm{v}_{\G^{\rho}_\gamma}\leq \norm{v}_{\G^{\rho}_{\gamma+\delta}} \text{ and } \norm{v}_{\G^{\rho+1}_\gamma}\leq\delta^{-1}\norm{v}_{\G^\rho_{\gamma+\delta}}
\end{equation} 
for any $\delta,\gamma,\rho>0$. Then iteratively (losing a fraction of the original $\epsilon$ at every step) we will be able to demonstrate the argument.
\end{proof}

\begin{proof}[Proof of Theorem \ref{thm:main_full}]
By Lemma \ref{lem:drift}, with probability $1-e^{-2\frac{\alpha \beta}{\nu^2}}$ we have that $\phi(t)-\nu W_t\geq 0$ for all $t\geq 0$. We may and will assume from now on that such a path has been chosen, and argue pathwise.

In view of iterating local well-posedness and energy estimates, set $\epsilon_j:=2^{-j}\epsilon$ and define 
\begin{equation*}
 \phi_0:=\phi+\epsilon,\qquad \phi_{j+1}:=\phi_j-\epsilon_{j+1},\quad j\geq 0. 
\end{equation*}
We note that
\begin{equation}\label{eq:phi-choice}
 \phi_0\geq\phi_j\geq\phi_{j+1}\geq \phi,\quad j\in\N.
\end{equation}

\noindent\emph{Base Case.} For $u_0$ as in \eqref{eq:init_cond1} we obtain a local solution $u\in C^0([0,T],\G^\sigma_{\phi_0})$ on a time interval $[0,T]$. Then from the embeddings \eqref{eq:G-emb-gen} it follows that $u\in C^0([0,T],\G^{\sigma}_{\phi_1})$, and $u$ also remains bounded in $C^0([0,T],\G^{\sigma+1}_{\phi_1})$: for $t\in[0,T]$ we have
\begin{equation*}
 \norm{u(t)}_{\G^{\sigma+1}_{\phi_1}}\leq\epsilon_1^{-1}\norm{u(t)}_{\G^\sigma_{\phi_0}}.
\end{equation*} 
Moreover, from the energy estimates in Proposition \ref{prop:en_est} we deduce the monotonicity of $\norm{u(t)}_{\G^\sigma_{\phi_1}}$, i.e.\ that
\begin{equation*}
 \norm{u(t)}_{\G^\sigma_{\phi_1}}\leq\norm{u_0}_{\G^\sigma_{\phi_1(0)}}\leq\norm{u_0}_{\G^\sigma_{\alpha+\epsilon}},\quad 0\leq t\leq T.
\end{equation*}

\noindent\emph{Inductive Step.} We iterate the above procedure: in the $j$-th step ($j\in\N$) we start with a solution $u\in C^0([0,jT],\G^\sigma_{\phi_j})$ defined on a time interval $[0,jT]$ with $$\norm{u(t)}_{\G^\sigma_{\phi_j}}\leq\norm{u_0}_{\G^\sigma_{\alpha+\epsilon}},\quad 0\leq t\leq jT.$$

By construction of $\phi_j$ (see \eqref{eq:phi-choice}) we can invoke our local well-posedness in Proposition \ref{prop:local}, and uniquely extend this solution in $C^0([0,(j+1)T],\G^\sigma_{\phi_{j}})$ to a time interval at least $[0,(j+1)T]$. As above, the Gevrey embedding \eqref{eq:G-emb-gen} guarantees that
\begin{equation*}
 \norm{u(t)}_{\G^{\sigma+1}_{\phi_{j+1}}}\leq\epsilon_j^{-1}\norm{u(t)}_{\G^\sigma_{\phi_j}}<+\infty,\quad jT\leq t\leq (j+1)T.
\end{equation*}
From the energy estimates in Proposition \ref{prop:en_est} we then deduce that for $t\in[jT,(j+1)T]$ we have
\begin{equation*}
 \norm{u(t)}_{\G^{\sigma}_{\phi_{j+1}}}\leq \norm{u(jT)}_{\G^{\sigma}_{\phi_{j+1}(jT)}}\leq\norm{u(jT)}_{\G^{\sigma}_{\phi_{j}(jT)}}\leq\norm{u_0}_{\G^{\sigma}_{\alpha+\epsilon}}.
\end{equation*}
This shows that we have constructed $u\in C^0([0,(j+1)T],\G^\sigma_{\phi_{j+1}})$ with
\begin{equation*}
 \norm{u(t)}_{\G^{\sigma}_{\phi_{j+1}}}\leq\norm{u_0}_{\G^{\sigma}_{\alpha+\epsilon}}\quad \text{ for }0\leq t\leq (j+1)T,
\end{equation*}
which concludes the inductive step.

This iteration yields a global solution $u\in C^0([0,\infty),\G^\sigma_\phi)$ as claimed. The monotonicity of $t\mapsto\norm{u(t)}_{\G^\sigma_\phi}$ is a direct consequence of the energy estimates in Proposition \ref{prop:en_est}.

\noindent\emph{Pathwise Uniqueness.} The pathwise uniqueness is inherent in our construction: For each $t\mapsto W_t(\omega)$ we have a contraction mapping principle (Proposition \ref{prop:local}) that yields a locally unique solution.
\end{proof}

Transforming back to the original variable $\theta(t)=\Gamma^{-1}u(t)$ we obtain the corresponding statement for the original equation \eqref{eq:SQG-prob}: By construction, whenever $\phi(t)-\nu W_t\geq 0$, the mapping $\Gamma^{-1}=e^{\nu W_t\abs{\nabla}^s}:\G^\sigma_\phi\to \dot{H}^{\sigma s}$ is bounded:
\begin{equation}
 \norm{e^{\nu W_t\abs{\nabla}^s}u(t)}_{\dot{H}^{\sigma s}}=\norm{e^{(\nu W_t-\phi(t))\abs{\nabla}^s}e^{\phi(t)\abs{\nabla}^s}u(t)}_{\dot{H}^{\sigma s}}\leq\norm{e^{\phi(t)\abs{\nabla}^s}u(t)}_{\dot{H}^{\sigma s}}=\norm{u(t)}_{\G^\sigma_\phi}.
\end{equation}
This demonstrates the following
\begin{corollary}\label{cor:main}
 Let $\theta_0$ with $$\norm{\theta_0}_{\G^\sigma_{\alpha+\epsilon}}\leq \frac{\nu^2}{2}-\beta$$ with $\alpha,\beta,\epsilon,\sigma >0$ as in Theorem \ref{thm:main_full} above. Then with probability at least $1-e^{-2\frac{\alpha \beta}{\nu^2}}$ we obtain a unique global solution $\theta\in C^0\dot{H}^{\sigma s}$ to the original stochastic SQG equation \eqref{eq:SQG-prob},
 \begin{equation*}
 d\theta +R^\perp\theta\cdot\nabla\theta dt=\nu\abs{\nabla}^s\theta\, dW_t,\quad \theta(0)=\theta_0.
\end{equation*}
Furthermore, the mapping $t\mapsto\norm{\theta(t)}_{\dot{H}^{\sigma s}}$ is pathwise monotonically decreasing with
\begin{equation*}
  \norm{\theta(t)}_{\dot{H}^{\sigma s}}\leq \norm{\theta_0}_{\G^\sigma_{\alpha+\epsilon}}.
\end{equation*}
\end{corollary}

\newpage
\bibliographystyle{alpha}
\bibliography{prob_sqg-refs.bib}

\end{document}